\documentclass[12pt]{article}
\usepackage{amsmath}
\usepackage{amssymb,latexsym}
\usepackage{amsthm}
\usepackage{bm}
\usepackage{bbm}
\newdimen\proofrulebreadth \proofrulebreadth=.05em
\newdimen\proofdotseparation \proofdotseparation=1.25ex
\newdimen\proofrulebaseline \proofrulebaseline=2ex
\newcount\proofdotnumber \proofdotnumber=3
\let\then\relax
\def\hfi{\hskip0pt plus.0001fil}
\mathchardef\squigto="3A3B
\newif\ifinsideprooftree\insideprooftreefalse
\newif\ifonleftofproofrule\onleftofproofrulefalse
\newif\ifproofdots\proofdotsfalse
\newif\ifdoubleproof\doubleprooffalse
\let\wereinproofbit\relax
\newdimen\shortenproofleft
\newdimen\shortenproofright
\newdimen\proofbelowshift
\newbox\proofabove
\newbox\proofbelow
\newbox\proofrulename
\def\shiftproofbelow{\let\next\relax\afterassignment\setshiftproofbelow\dimen0 }
\def\shiftproofbelowneg{\def\next{\multiply\dimen0 by-1 }%
\afterassignment\setshiftproofbelow\dimen0 }
\def\setshiftproofbelow{\next\proofbelowshift=\dimen0 }
\def\setproofrulebreadth{\proofrulebreadth}

\def\prooftree{
\ifnum  \lastpenalty=1
\then   \unpenalty
\else   \onleftofproofrulefalse
\fi
\ifonleftofproofrule
\else   \ifinsideprooftree
        \then   \hskip.5em plus1fil
        \fi
\fi
\bgroup
\setbox\proofbelow=\hbox{}\setbox\proofrulename=\hbox{}%
\let\justifies\proofover\let\leadsto\proofoverdots\let\Justifies\proofoverdbl
\let\using\proofusing\let\[\prooftree
\ifinsideprooftree\let\]\endprooftree\fi
\proofdotsfalse\doubleprooffalse
\let\thickness\setproofrulebreadth
\let\shiftright\shiftproofbelow \let\shift\shiftproofbelow
\let\shiftleft\shiftproofbelowneg
\let\ifwasinsideprooftree\ifinsideprooftree
\insideprooftreetrue
\setbox\proofabove=\hbox\bgroup$\displaystyle 
\let\wereinproofbit\prooftree
\shortenproofleft=0pt \shortenproofright=0pt \proofbelowshift=0pt
\onleftofproofruletrue\penalty1
}

\def\eproofbit{
\ifx    \wereinproofbit\prooftree
\then   \ifcase \lastpenalty
        \then   \shortenproofright=0pt  
        \or     \unpenalty\hfil         
        \or     \unpenalty\unskip       
        \else   \shortenproofright=0pt  
        \fi
\fi
\global\dimen0=\shortenproofleft
\global\dimen1=\shortenproofright
\global\dimen2=\proofrulebreadth
\global\dimen3=\proofbelowshift
\global\dimen4=\proofdotseparation
\global\count255=\proofdotnumber
$\egroup  
\shortenproofleft=\dimen0
\shortenproofright=\dimen1
\proofrulebreadth=\dimen2
\proofbelowshift=\dimen3
\proofdotseparation=\dimen4
\proofdotnumber=\count255
}

\def\proofover{
\eproofbit 
\setbox\proofbelow=\hbox\bgroup 
\let\wereinproofbit\proofover
$\displaystyle
}%
\def\proofoverdbl{
\eproofbit 
\doubleprooftrue
\setbox\proofbelow=\hbox\bgroup 
\let\wereinproofbit\proofoverdbl
$\displaystyle
}%
\def\proofoverdots{
\eproofbit 
\proofdotstrue
\setbox\proofbelow=\hbox\bgroup 
\let\wereinproofbit\proofoverdots
$\displaystyle
}%
\def\proofusing{
\eproofbit 
\setbox\proofrulename=\hbox\bgroup 
\let\wereinproofbit\proofusing
\kern0.3em$
}

\def\endprooftree{
\eproofbit 
  \dimen5 =0pt
\dimen0=\wd\proofabove \advance\dimen0-\shortenproofleft
\advance\dimen0-\shortenproofright
\dimen1=.5\dimen0 \advance\dimen1-.5\wd\proofbelow
\dimen4=\dimen1
\advance\dimen1\proofbelowshift \advance\dimen4-\proofbelowshift
\ifdim  \dimen1<0pt
\then   \advance\shortenproofleft\dimen1
        \advance\dimen0-\dimen1
        \dimen1=0pt
        \ifdim  \shortenproofleft<0pt
        \then   \setbox\proofabove=\hbox{%
                        \kern-\shortenproofleft\unhbox\proofabove}%
                \shortenproofleft=0pt
        \fi
\fi
\ifdim  \dimen4<0pt
\then   \advance\shortenproofright\dimen4
        \advance\dimen0-\dimen4
        \dimen4=0pt
\fi
\ifdim  \shortenproofright<\wd\proofrulename
\then   \shortenproofright=\wd\proofrulename
\fi
\dimen2=\shortenproofleft \advance\dimen2 by\dimen1
\dimen3=\shortenproofright\advance\dimen3 by\dimen4
\ifproofdots
\then
        \dimen6=\shortenproofleft \advance\dimen6 .5\dimen0
        \setbox1=\vbox to\proofdotseparation{\vss\hbox{$\cdot$}\vss}%
        \setbox0=\hbox{%
                \advance\dimen6-.5\wd1
                \kern\dimen6
                $\vcenter to\proofdotnumber\proofdotseparation
                        {\leaders\box1\vfill}$%
                \unhbox\proofrulename}%
\else   \dimen6=\fontdimen22\the\textfont2 
        \dimen7=\dimen6
        \advance\dimen6by.5\proofrulebreadth
        \advance\dimen7by-.5\proofrulebreadth
        \setbox0=\hbox{%
                \kern\shortenproofleft
                \ifdoubleproof
                \then   \hbox to\dimen0{%
                        $\mathsurround0pt\mathord=\mkern-6mu%
                        \cleaders\hbox{$\mkern-2mu=\mkern-2mu$}\hfill
                        \mkern-6mu\mathord=$}%
                \else   \vrule height\dimen6 depth-\dimen7 width\dimen0
                \fi
                \unhbox\proofrulename}%
        \ht0=\dimen6 \dp0=-\dimen7
\fi
\let\doll\relax
\ifwasinsideprooftree
\then   \let\VBOX\vbox
\else   \ifmmode\else$\let\doll=$\fi
        \let\VBOX\vcenter
\fi
\VBOX   {\baselineskip\proofrulebaseline \lineskip.2ex
        \expandafter\lineskiplimit\ifproofdots0ex\else-0.6ex\fi
        \hbox   spread\dimen5   {\hfi\unhbox\proofabove\hfi}%
        \hbox{\box0}%
        \hbox   {\kern\dimen2 \box\proofbelow}}\doll%
\global\dimen2=\dimen2
\global\dimen3=\dimen3
\egroup 
\ifonleftofproofrule
\then   \shortenproofleft=\dimen2
\fi
\shortenproofright=\dimen3
\onleftofproofrulefalse
\ifinsideprooftree
\then   \hskip.5em plus 1fil \penalty2
\fi
}

\usepackage[all,cmtip]{xy}
\input{diagxy}
\CompileMatrices       
\usepackage{url}

\newcommand{\C}{\ensuremath{\mathbb{C}}}
\newcommand{\Cc}{\ensuremath{\mathbb{C}_\times}}

\newcommand{\B}{\ensuremath{\mathbb{B}}}
\newcommand{\N}{\ensuremath{\mathbb{N}}}
\newcommand{\T}{\ensuremath{\mathbb{T}}}

\newcommand{\psh}[1]{\ensuremath{\mathsf{Set}^{#1^{\mathrm{op}}}}}
\newcommand{\op}[1]{\ensuremath{#1^{\mathrm{op}}}}
\newcommand{\Set}{\ensuremath{\mathsf{Set}}}
\newcommand{\cSet}{\ensuremath{\mathsf{cSet}}}

\newcommand{\pocorner}[1][dr]{\save*!/#1+2pc/#1:(2,-2)@^{|-}\restore}
\newcommand{\pbcorner}[1][dr]{\save*!/#1-1.2pc/#1:(-1,1)@^{|-}\restore}
\newcommand{\y}{\ensuremath{\mathsf{y}}} 
\newcommand{\Hom}{\ensuremath{\mathrm{Hom}}}
\renewcommand{\hom}{\ensuremath{\mathrm{Hom}}}
\renewcommand{\L}{\ensuremath{\mathcal{L}}}
\newcommand{\R}{\ensuremath{\mathcal{R}}}
\newcommand{\LL}{\ensuremath{\mathsf{L}}}
\newcommand{\RR}{\ensuremath{\mathsf{R}}}
\newcommand{\eval}{\ensuremath{\mathrm{eval}}}

\newcommand{\hook}{\ensuremath{\hookrightarrow}}
\newcommand{\mono}{\ensuremath{\rightarrowtail}}
\newcommand{\arr}{\ensuremath{\rightarrow}}

\newcommand{\Tinf}{\ensuremath{T_\infty}}

\newcommand{\I}{\ensuremath{\mathrm{I}}}

\newcommand{\pA}{\ensuremath{A^\I}}

\newcommand{\G}{\ensuremath{\Gamma}}

\newcommand{\type}{\mathsf{type}}       
\newcommand{\types}[2]{#1 \vdash #2\ \type}
\newcommand{\Gtypes}[1]{\types{\Gamma}{#1}}

\newcommand{\terms}[2]{#1 \vdash #2}

\newcommand{\Id}{\ensuremath{\mathsf{Id}}}
\newcommand{\id}[1]{\Id_{#1}}
\newcommand{\refl}{\mathsf{refl}}

\newtheorem{theorem}{Theorem}
\newtheorem*{theorem*}{Theorem}
\newtheorem{proposition}[theorem]{Proposition} 
\newtheorem{lemma}[theorem]{Lemma}
\newtheorem{corollary}[theorem]{Corollary} 

\theoremstyle{remark}
\newtheorem{remark}[theorem]{Remark} 
\newtheorem*{remarks*}{Remarks}

\theoremstyle{definition}
\newtheorem{definition}[theorem]{Definition}

\begin{document}

\title{A cubical model of homotopy type theory\thanks{
Notes from a series of lectures given in May-June, 2016, in the Stockholm Logic Seminar and the Topological Actitivies Seminar.  Thanks to Erik Palmgren for providing the opportunity and to the members of the Stockholm Logic Group for a very enjoyable stay.}
}
\author{Steve Awodey
}
\date{Stockholm, 21 June 2016}

\maketitle
\noindent
The main goal of these notes is to prove the following:

\begin{theorem*}
There is an algebraic weak factorization system $(\mathsf{L}, \mathsf{R})$ on the category of cartesian cubical sets such that for any $\mathsf{R}$-object $A$, the factorization of the diagonal map,
\[
A \to A^\I \to A\times  A\,,
\]
determined by the 1-cube $\I$, is an $(\mathsf{L}, \mathsf{R})$-factorization.
\end{theorem*}

\noindent It follows that there is a cubical model of homotopy type theory in which the identity type of a type $A$ is taken to be the path-object $A^\I$.

We begin by reviewing the basic idea of homotopical semantics of type theory in weak factorization systems, including the somewhat technical issue of coherence that motivates the use of algebraic weak factorization systems. We then consider cubical sets and construct the desired algebraic weak factorization system.

\section{The basic homotopical interpretation}

\begin {definition}
A \emph{weak factorization system} on a category \C\ consists of two classes of arrows,
\[
\xymatrix{
\L\ \ar@{_{(}->}[r] & \C_1 & \ar@{_{(}->}[l]\ \R
}
\]
satisfying the following conditions:
\begin{enumerate}
\item Every arrow $f : X\to Y$ in \C\ factors as a left map followed by a right map,
\[
\xymatrix{
 X \ar[rr]^{f} \ar[rd]_{\L} && Y\\
& \cdot \ar[ru]_{\R}  &
}
\]
\item Given any commutative square
\[
\xymatrix{
A \ar[d]_{\L} \ar[r] & B \ar[d]^{\R} \\
C \ar[r] \ar@{..>}[ru] & D
}
\]
with an $\L$-map on the left and an $\R$-map on the right, there is a ``diagonal filler'' as indicated, making both triangles commute.

\item Each of the classes $\L$ and $\R$ is closed under retracts in the arrow category $\C^{\rightarrow}$.
\end{enumerate}
\end{definition}

Examples include (i) Groupoids (or categories), with $\R$ = isofibrations and $\L$ = injective equivalences; (ii) Simplicial sets, with $\R$ = Kan fibrations and $\L$ = acyclic cofibrations.  Since a Quillen model structure on a category by definition involves two interrelated such weak factorization systems, this provides many examples as well as the basic homotopical intuition. In a WFS, we may think of the \R-maps as ``fibrations'', i.e.\ good families of objects indexed by the codomain.  The basic idea of the homotopy interpretation is to use these as the dependent types.

\subsection{Interpreting $\Id$-types}

Let $\C$ be a category with finite limits and a WFS.  Closed types are interpreted as $\R$-objects $A$, i.e.\ those for which $A\arr 1$ is in $\R$.  Dependent types $\types{x:A}{B}$ are interpreted as \R-maps $B\arr A$. Terms $\terms{x:A}{b:B}$ are sections $b: A\arr B$ of the chosen \R-map $B\arr A$.

The formation rule for \Id-types says that each type has an identity type:
\begin{equation}\label{rule:idform}\tag{$\Id$-Form}
\begin{prooftree}
A\ \type 
\justifies
x,y: A\ \vdash\ \id{A}(x,y)\ \type
\end{prooftree}
\end{equation}
We  model this by factoring the diagonal map of (the object interpreting) $A$ as an \L-map followed by an \R-map, using axiom 1 for the WFS:
\[
\xymatrix{
& \ \id{A} \ar[rd] &\\
A \ar[rr] \ar[ru] && A\times A
}
\]
The \R-map $\id{A} \arr A\times A$ interprets the dependent type $x,y: A \vdash \id{A}(x,y)\ \type$.
Th \L-factor $A\to \id{A}$ interprets the reflexivity term $\refl(x)$ in the \Id-introduction rule:
\begin{equation}\label{rule:idintro}\tag{$\Id$-Intro}
x:A\  \vdash\ \refl(x) : \id{A}(x,x) 
\end{equation}
The  $\Id$-elimination rule has the form:
\begin{equation}\label{rule:idelim}\tag{$\Id$-Elim}
\begin{prooftree}
x,y: A, z:\id{A}(x,y)\vdash B(x,y,z)\ \type, \qquad
x: A \vdash b(x) : B(x,x,\refl(x))
\justifies
x,y: A, z:\id{A}(x,y)\vdash \mathsf{J}_b(x,y,z) : B(x,y,z)
\end{prooftree}
\end{equation}
with associated computation rule:
\begin{equation}\label{rule:idcomp}\tag{$\Id$-Comp}
\mathsf{J}_b(x,x,\refl(x)) = b(x) : B(x).
\end{equation}
The data above the line in $\Id$-elimination are interpreted as a commutative square as on the outside of the following diagram:
\begin{equation}\label{diag:idelim}
\xymatrix@=3em{
A \ar[d]_{\refl} \ar[r]^b & B \ar[d] \\
\id{A} \ar[r]_{=} \ar@{..>}[ru]_{\mathsf{J}_b} & \id{A} 
}
\end{equation}
Since $\refl$ is an \L-map, and $B\arr \id{A}$ is an \R-map (as the interpretation of a dependent type), there is a diagonal filler $\mathsf{J}_b$ as indicated.  The commuting of the lower triangle means that $\mathsf{J}_b$ is a section of $B\arr \id{A}$ and thus a term of the type required by the conclusion of the \Id-elimination rule.  The commuting of the upper triangle is exactly the $\mathsf{J}$-computation rule. (See \cite{AW}.)

\subsection{Coherence}

The interpretation just sketched is required to respect the result of substituting a term into a context, since the rules have this property.  Substitution into dependent types is interpreted as pullback, and substitution into terms as (roughly) composition. There are then three separate issues involved in giving a strict interpretation of type theory with \Id-types, and all three are called ``coherence'':

\begin{enumerate}

\item Using the fact that \R-maps are closed under retracts, one can show that they are also stable under pullback along any map, so the interpretation of dependent types as \R-maps is compatible with the interpretation of substitution as pullback.  However, the fact that the pullback operation is defined only up to isomorphism means that the interpretation must be ``strictified'' in order to model substitution strictly.  This is a known issue in the semantics of dependent type theory, with equally well-known solutions (including a recent one due to Lumsdaine and Warren \cite{LW}), and will not concern us further here.

\item The choice of factorization of the diagonal,
\[
\xymatrix{
& \ \id{A} \ar[rd] &\\
A \ar[rr] \ar[ru] && A\times A\,,
}
\]
must be stable under pulback.  Specifically, if $A$ is a type in context $\Gtypes{A}$ interpreted as an \R-map $A\arr\G$, then there is a factorization of the diagonal over $\G$ of the form
\[
\xymatrix{
& \ \id{A} \ar[rd] &\\
A \ar[rr] \ar[ru] \ar[rd] && \ar[ld] A\times_{\G} A\,,\\
& \G &
}
\]
Pulling back along any $f : \Delta \to \G$ preserves the diagonal, but not necessarily the \L-\R\ factorization, 
\[
\xymatrix{
& \ f^*\id{A} \ar[rd] &\\
f^*A \ar[rr] \ar[ru] \ar[rd] && \ar[ld] f^*A\times_{\Delta} f^*A\,.\\
& \Delta &
}
\]
Choosing an \L-\R\ factorization of the pulled-back diagonal gives an interpretation of $\id{f^*A}$ that need not agree (even up to isomorphism) with $f^*\id{A}$.  A choice of factorizations, for all diagonals, that respects pullback in this sense is said to be \emph{stable}.  One way such a stable choice of factorizations can arise is when it is determined by exponentiating by a fixed ``interval'' object $\I$, so that $ \id{A}=A^\I$. This is what happens, for example, in the groupoid model, where as an interval one can take the groupoid with two objects and two, mutually inverse, non-identity arrows.

\item Assuming a stable choice of factorizations of the diagonal, we have made a choice of diagonal fillers $\mathsf{J}$ in order to interpret the $\Id$-elimination rule.  Again, there is no reason why these choices of diagonal fillers should ``respect substitution'' in the way required for the interpretation of type theory. 
More specifically, given a diagonal filling problem as on the right below, and a square on the left with $g\in\L$,
\begin{equation}\label{diagram:coherencefillers}
\xymatrix@=3em{
A' \ar[d]_g \ar[r] & A  \ar[d] \ar[r] & \ar[d] C \\
B' \ar[r]_{f} \ar@{..>}[rru]^{\psi\ \ } & B\ar@{..>}[ru]_{\phi} \ar[r] & D
}
\end{equation}
we may have the two different diagonal fillers for the outer filling problem, namely $\psi$ and $\phi\circ f$.  Under certain conditions, we want  these two solutions to be equal.
This leads to a strengthening of the notion of weak factorization system to what is now called an \emph{algebraic} weak factorization system, which implies the existence of such \emph{natural} choices of diagonal fillers. (The idea of using this to solve this particular coherence problem is due to R.~Garner.)

\end{enumerate}

\begin{definition} A \emph{functorial factorization} on a category \C\ is a functor
\[
(L,E,R) : \C^{\arr} \to \C^{\arr\cdot\arr}
\]
taking each arrow $f : X \to Y$ to a factorization $f = R(f)\circ L(f)$,
\[
\xymatrix{
 X \ar[rr]^{f} \ar[rd]_{L(f)} && Y\\
& E(f) \ar[ru]_{R(f)}  &
}
\]
in a functorial way.  Specifically, given any $h : f \to f'$ in $\C^{\arr}$ we have a commutative diagram:
\[
\xymatrix{
X \ar[dd]_{f} \ar[rd]^{L(f)} \ar[rrr]^{h_0} &&& \ar[dd]^{f'} \ar[ld]_{L(f')} X'\\
   & E(f) \ar[r] \ar[ld]^{R(f)} & E(f') \ar[rd]_{R(f')} & \\
Y  \ar[rrr]_{h_1} &&&  Y'
}
\]
and we write $E(h) : E(f)\to E(f')$ for the evident map.
\end{definition}

Observe that, for each fixed $Y$, the functorial factorization determines an endofunctor, 
\[
R : \C/Y \to \C/Y
\]
taking $f:X\to Y$ to $R(f) :E(f) \to Y$, and that this endofunctor is pointed by $L : 1 \to R$ (abusing notation slightly).

Dually, for each fixed $X$, the functorial factorization determines an endo\-functor, 
\[
L : X/\C \to X/\C
\]
taking $f:X\to Y$ to $L(f) : X\to E(f)$, and that this endofunctor is copointed by $R : L \to 1$.

\begin{definition}
An \emph{algebraic weak factorization system} on \C\ consists of a functorial factorization $(L,E,R)$ together with:
\begin{enumerate}
\item a  multiplication $\mu : R^2 \to R$ making $(R, \mu, L)$ a monad,
\item a comultiplication $\nu : L \to L^2$ making $(L, \nu, R)$ a comonad.
\end{enumerate}
Some authors also a require distributive law for the monad over the comonad, however we shall not need this.
\end{definition}

\begin{remark}\label{AWFSisWFS} Let us show that an AWFS determines a WFS.  The factorization axiom is satisfied by the functorial factorization $f = R(f)\circ L(f)$.  We then know that $R(f)$ is an $R$-algebra and $L(f)$ is an $L$-coalgebra by the laws of monads.  Suppose given a diagonal filling problem such as  the outer square below, in which $f$ is an $L$-coalgebra and $g$ is an $R$-algebra:
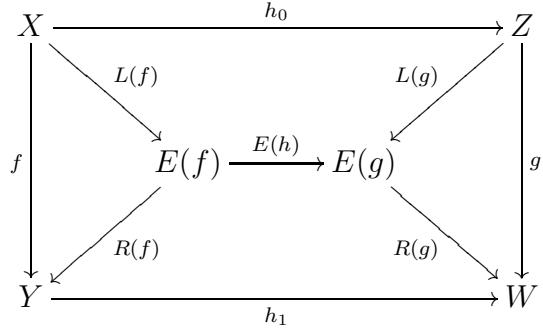
\[
\xymatrix@=3em{
X \ar[dd]_{f} \ar[rd]^{L(f)} \ar[rrr]^{h_0} &&& \ar[dd]^{g} \ar[ld]_{L(g)} Z\\
   & E(f) \ar[r]^{E(h)} \ar[ld]^{R(f)} & E(g) \ar[rd]_{R(g)} & \\
Y  \ar[rrr]_{h_1} &&&  W
}
\]
Applying the factorizations of $f$ and $g$, we obtain an $L$-coalgebra structure map $\phi : Y\to E(f)$ and an $R$-algebra structure map $\psi : E(g)\to Z$. We can then set $ j = \psi\circ E(h)\circ \phi$ to obtain the required diagonal filler $j : Y \to Z$.  Finally, to ensure closure under retracts we let $\R$ be the retract closure of the $R$-algebras and $\L$ the retract closure of the $L$-coalgebras.  The factorization axiom still holds trivially, and the filling axiom is also easily seen to still hold.  Thus every AWFS determines a WFS with the left and right classes being the retracts of the $L$- and $R$- (co)algebras respectively.
\end{remark}

\begin{remark} A morphism of $L$-coalgebras $h : (f',\phi')\to (f,\phi)$ is a commutative square $fh_0 = h_1f'$ such that $E(h)\circ\phi' = \phi\circ h_1$,
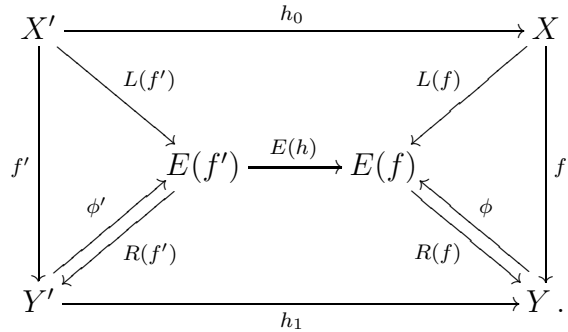
\[
\xymatrix@=3em{
X' \ar[dd]_{f'} \ar[rd]^{L(f')} \ar[rrr]^{h_0} &&& \ar[dd]^{f} \ar[ld]_{L(f)} X\\
   & E(f') \ar[r]^{E(h)} \ar[ld]^{R(f')} & E(f) \ar[rd]_{R(f)} & \\
Y'  \ar[rrr]_{h_1} \ar@<1ex>[ur]^{\phi'} &&& \ar@<-1ex>[ul]_{\phi} Y\,.
}
\]
It is easy to see that the naturality condition for diagonal fillers mentioned in \eqref{diagram:coherencefillers} is satisfied for the fillers constructed algebraically as in Remark \ref{AWFSisWFS}, when the left-hand square is a morphism of $L$-algebras in this sense.
\end{remark}

We summarize the result of this section with the following (cf.~\cite{GvdB}).

\begin{proposition}
Let \C\ be a category with finite limits, an algebraic weak factorization system, and a stable choice of factorizations for all diagonal maps.  Then \C\ admits a model of type theory with \Id-types.
\end{proposition}

\section{Cubical sets}

Our goal is to make an algebraic weak factorization system on the cubical sets, but let us first recall \emph{why} the category of cubical sets is a good setting for a model of \Id-types.  The basic examples of WFSs and homotopy are topological spaces and simplicial sets, but these models are not entirely satisfactory, at least from a logical point of view.  In particular, they seem to lack the combinatorial character that would make them \emph{constructive}.  The cubical approach seems to be better suited to giving such a constructive treatment, as has recently been shown by T.~Coquand and his coworkers (see \cite{BCH}).  It is worth noting in this connection that the goal here is \emph{not} to give a constructive interpretation of the univalence axiom, as Coquand et al.\ have done, but only to investigate a (particular) cubical model as arising from an algebraic weak factorization system.

By way of motivation for using cubical sets, recall that we will interpret the \Id-type of an object $A$ using a factorization of the diagonal map,
\begin{equation}\label{diag:pathfact}
\xymatrix{
& \ PA \ar[rd] & \\
A \ar[rr] \ar[ru] && A\times A\,.
}
\end{equation}
We regard $PA$ as an abstract  ``pathspace'' $A^\I$.  If there is an ``interval object'' $\I$ with two points $1\rightrightarrows\I$, then exponentiating $A$ by the structure:
\[
\xymatrix{
& \ar[ld] \I  & \\
1 && \ar[ll] \ar[ul] 1+1\,
}
\]
will result in a factorization of the form \eqref{diag:pathfact} (without respect to any WFS).  Moreover, this factorization will automatically be stable under pullback.

The idea that identity proofs are paths and that $\id{A} = A^\I$ is supported by the fact that dependent types have the \emph{path-lifting property}: given $a_0,a_1:A$  and $p:\id{A}(a_0,a_1)$ and $\types{x:A}{B(x)}$  and $b_0:B(a_0)$, there is a distinguished $b_1:B(a_1)$, called the \emph{transport} of $b_0$ along $p$ and denoted $p_*(b_0)$.  This term $p_*(b_0)$ can easily be found by ``path induction'' (i.e.\ \Id-elimination) on $p:\id{A}(a_0,a_1)$.
Indeed, there is even a path $\tilde{p} : \id{B}(b_0,b_1)$, where $B=\sum_{x:A}B(x)$, such that $\pi_1(\tilde{p})=p$.
\[
\xymatrix{
B \ar[dd]_{\pi_1} & B(a_0) \ar[r]^{p_*} & B(a_1) \\
   	& b_0 \ar@{..>}[r]^{\tilde{p} } & p_*b_0  \\
A  	&	a_0 \ar[r]_{p} &  a_1
}
\]
Note that this says exactly that the endpoint inclusion $0 : 1\to \I$ has the left-lifting property with respect to the dependent projection $\pi:B\to A$,
\[
\xymatrix@=3em{
1 \ar[d]_{0} \ar[r]^{b_0} & B \ar[d]^{\pi} \\
\I \ar[r]_{p} \ar@{..>}[ru]_{\tilde{p}} & A\,.
}
\]
Thus the endpoint inclusion $0 : 1\to \I$ should be an $\L$-map, since the dependent types should be $\R$-maps.

An important consequence of the proposed stipulation $\id{A} = A^\I$ is that the iterated or ``higher'' \Id-types can be interpreted as ``cube types'',
\begin{align*}
\id{A}\ &=\ A^\I \\
\id{\id{A}}\ &=\ (A^\I)^\I\ \cong\ A^{\I\times\I} \\
\id{\id{\dots\id{A}}}\ &\cong\ A^{\I\times\dots\times\I}
\end{align*}
In this way, at least some of the operations on the higher \Id-types (reflecting the higher algebra of $\infty$-groupoids) can be represented by algebraic operations on the cubes, in the form $\I^n \to \I^m$.

\subsection{The category of cubical sets}

The objects $\I^n = \I\times \dots\times \I$ thus play a special role in the  interpretation of \Id-types.  The category of cubical sets gives special prominence to these objects, and to the interval $\I$ in particular.  In a precise sense, it is the free cocomplete category generated by the finite cubes $1,\ \I,\ \I\times\I, \dots, \I^n, \dots$ (and the particular version that we shall use is the free \emph{topos} generated by an ``interval object'': $0,1 : 1\rightrightarrows \I,\ \cdot\vdash 0\neq 1$).

\begin{definition}
The category \cSet\ of \emph{(cartesian) cubical sets} is the presheaf category
\[
\cSet\ =\ \psh{\C_\times}
\]
on the category $\C_\times$ of \emph{(cartesian) cubes}, defined equivalently as:
\begin{enumerate}
\item the free category with finite products on an \emph{interval object} $0,1 : 1\rightrightarrows \I$,
\item the opposite category,
\[
\C_\times = \op{\B},
\]
of the category $\B$ of finite strictly bipointed sets and bipointed functions.  The objects may be taken to be sets of the form $$[n] = \{\bot, x_1, \dots, x_n, \top\}$$ and maps $f : [m]\to [n]$ are functions that preserve the distinguished elements $\bot, \top$.

\item  the syntactic category $\C(\T)$ of the algebraic theory $\T$ with just two constant symbols $\bot, \top$ (and no equations).
\end{enumerate}
The equivalence of these three specifications of $\C_\times$ is an easy application of Lawvere duality.
\end{definition}

\begin{remark} There are many other notions of cube category and cubical sets in the literature. See \cite{awodey:cubes} for  comparisons of some.
\end{remark}

Next, consider the representable objects,
\[
\y : \C_\times \hook \cSet
\]
and let
\[
\I^n :=\  \y([n]) = \Hom_{\C_\times}\!(\,-\,, [n]).
\]
We can regard the objects $\I^n$ as the ``geometric $n$-cubes''; indeed we have $\I^n \cong \I\times\dots\times\I$, because the Yoneda embedding preserves products.  Note that by the Yoneda lemma, for any cubical set $X$,  maps $\I^n \to X$ correspond uniquely to elements of the set $X_n = X([n])$, which are called the ``$n$-cubes of $X$''.

\begin{proposition}\label{prop:pathobject}
For any cubical set $X$, the pathobject $X^\I$ is the ``shift by one dimension'' of $X$, 
\[
(X^\I)_n \cong X_{(n+1)}\,.
\]
\end{proposition}

\begin{proof}
\[
\begin{split}
(X^\I)_n\ &\cong\ \hom(y[n], X^\I)\ \cong\ \hom(\I^n, X^\I)\ \cong\ \hom(\I^n\times \I, X)\\
&\cong\ \hom(\I^{n+1}, X)\cong\ \hom(y[n+1], X)\ \cong\ X_{(n+1)}.
\end{split}
\]
\end{proof}

This combinatorial description of the pathobject $X^\I$ as a shift in dimension has consequences for the model of type theory; in particular, it  leads to certain equations holding strictly that are usually satisfied only weakly or ``up to homotopy''.

\begin{corollary}\label{cor:rightadjpath}
The pathobject functor $X \mapsto X^\I$ has a \emph{right} adjoint.
\end{corollary}
\begin{proof}
The functor $X \mapsto X^\I$ is given by precomposition with the ``successor'' functor $S : \Cc\to\Cc$ with $S[n] = [n+1]$, since $X^\I([n]) = X([n+1]) = X(S[n]) = (S^*(X))([n])$.  But precomposition always has both left and right adjoints $S_! \dashv S^* \dashv S_*$, the right one of which can be calculated as:
\[
S_*(X)_n \cong \hom(y[n], S_*X) \cong \hom(S^*(y[n]), X)\ \cong \hom(\Cc(S(-), [n]), X).
\]
\end{proof}

We need the following fact in order to calculate the right adjoint further. We mention that a similar fact holds for the generic object in the object classifier topos $\Set^{\mathsf{Fin}}$, and in the Schanuel topos $\mathrm{Sh}(\mathrm{Aut}(\N))$, and is used to give an algebraic treatment of variable binding in the theory of ``abstract higher-order syntax''.

\begin{lemma}\label{lemma:binomial}
For the 1-cube $\I$ in $\cSet$, we have $\I^\I\ \cong\ \I+1$.
\end{lemma}
\begin{proof}
For any $[n] \in \Cc$ we have:
\[
(\I^\I)_n \cong \I_{(n+1)} \cong \Hom(\I^{(n+1)},\I)\cong \Cc([n+1],[1])\cong \B([1], [n+1])\cong n+3.
\]
On the other hand,
\[
(\I+1)_n \cong \I_n + 1_n \cong \Hom(\I^n, \I) + 1 \cong \B([1],[n]) +1 \cong (n+2) +1.
\]
The isomorphism is natural in $n$.
\end{proof}

\begin{definition}
Let us write
\[
X_\I = S_*(X)
\]
for the right adjoint of the path object functor $X^\I = S^*X$.
\end{definition}

\begin{corollary}
We have the following calculation for the right adjoint $X_\I$:
\begin{align*}
(X_\I)_n &\cong \hom(\I^n, X_\I) \\
&\cong \hom((\I^n)^\I, X) \\
&\cong \hom((\I^\I)^n, X) \\
&\cong \hom((\I+1)^n, X) \\
&\cong \hom(\I^n + C^n_{n-1}\I^{n-1} + \dots + C^n_{1}\I+1, X) \\
&\cong X_n \times X_{n-1}^{C^n_{n-1}} \times \dots \times X_1^{C^n_{1}}\times X_0,
\end{align*}
where $C^n_{k} = \binom{n}{k}$ is the usual binomial coefficient.
\end{corollary}

\subsection{Box filling}

We now turn to the determination of the right maps for our AWFS.    We know that they should at least have the path-lifting property, since that is forced by the rules of type theory.  We also know that for any type $A$, the interpretation of the \Id-type should be a right map into $A\times A$, since it is a dependent type. Finally, we want to use the pathobject $A^\I$ as the interpretation of $\Id{A}$ in view of the foregoing considerations. These constraints lead us directly to the box-filling condition.

\begin{proposition}\label{prop:pathliftboxfill}
$A^\I \to A\times A$ has path lifting if and only if $A$ has (2-)box filling.
\end{proposition}

Before giving the proof, we require some conventions that will be useful later on.
Diagramatically, path lifting for $A^\I \to A\times A$ means that for any (outer) square of the form
\begin{equation}\label{diagram:pathliftpath}
\xymatrix{
1 \ar[d] \ar[r] & A^\I \ar[d] \\
\I \ar[r] \ar@{..>}[ru] & A\times A
}
\end{equation}
there is a diagonal filler as indicated.  There are two maps $0,1 : 1\rightrightarrows \I$, and the path lifting condition is required with each of these cases occurring on the left.

An \emph{open 2-box} in the 2-cube $\I^2$ is by definition a subobject $$\sqcup^2_{j,e} \mono \I^2$$ obtained as the union, in the poset $\mathrm{Sub}(\I^2)$, of all the face maps $\alpha^d_i : \I \mono \I^2$ but one,
\[
\sqcup^2_{j,e}  = \bigcup_{(i,d)\neq (j,e)} \alpha_i^d.
\]
The index $j = 1, 2$ is the coordinate in which the box is open, while $e = 0,1$ indicates which face of the box is missing, bottom or top.   Because we are in the symmetric situation, where the product $\I\times \I$ can be twisted, it suffices to  consider only  boxes that are open in the first coordinate, since the others can be constructed from those.  Thus we may omit the index $j$, writing $\sqcup^2_e$ for $\sqcup^2_{1,e}$. Moreover, let us write $\sqcup^2 = {\sqcup}^2_1$ and $\sqcap^2 = {\sqcup}^2_0$.  The \emph{(upper and lower) open $n$-boxes} $\sqcup^n, \sqcap^n \mono \I^n$ are defined analogously.

The open 2-box can be constructed as the dotted arrow in the following pushout diagram, in which we write $\partial\I = 1+1 \mono \I$ for the boundary of the 1-cube $\I$.
\[
\xymatrix{
1\times \partial\I \ar[d] \ar[r]  & 1\times \I \ar[d] \ar@/{}^{1pc}/[rdd] \\
\I \times \partial\I \ar[r] \ar@/{}_{1pc}/[rrd] &  \sqcup^2 \pocorner \ar@{>.>}[rd] \\
 &&  \I\times\I 
}
\]
This is the upper open box determined by the case where $1 : 1\to\I$ is on the far left; the case of $0: 1\to\I$ determines the lower open box $\sqcap^2 \mono \I^2$, but we will not always mention this case separately.  

\begin{definition}\label{def:boxfilling}
A cubical set $A$  has \emph{$n$-box filling} if every map to it from an open upper $n$-box extends to the whole $n$-cube,
\[
\xymatrix{
\sqcup^n \ar@{ >->}[d] \ar[r]  & A  \\
\I^n \ar@{ ..>}[ur] 
}
\]
and similarly for the lower box $\sqcap^n \mono \I^n$.  
A map $B\to A$ has \emph{$n$-box filling} if every commutative square of the following form has a diagonal filler,
\[
\xymatrix{
\sqcup^n \ar@{ >->}[d] \ar[r]  & B \ar[d]  \\
\I^n \ar@{ ..>}[ur] \ar[r] & A
}
\]
and similarly for the lower box $\sqcap^n \mono \I^n$.  
\end{definition}

\begin{proof}[Proof of the Proposition]
The 2-box filling condition is clearly equivalent to saying that given any maps $a$ and $b$ commuting with the span in the upper-left corner of the following diagram, there exists a 2-cube $c :  \I\times\I \to A$  making the whole diagram commute.  
\begin{equation}\label{diagram:boxfillingrev}
\xymatrix{
1\times \partial\I \ar[d] \ar[r]  & 1\times \I \ar[d] \ar@/{}^{1pc}/[rdd]^{a} \\
\I \times \partial\I \ar[r] \ar@/{}_{1pc}/[rrd]_{b} & \I\times\I \ar@{.>}[rd]|{c} \\
 &&  A
}
\end{equation}
This formulation eliminates the pushout and replaces the open box by a decomposition.

Now let us rewrite diagram \eqref{diagram:pathliftpath} with the projection from the path object $A^\I \to A\times A$  replaced by
\[
A^i : A^\I \to A\times A \cong A^{1+1} = A^{\partial\I} 
\] 
where $i := [0,1] : \partial \I = 1+1 \to \I$ is the copair, to give:
\begin{equation*}\label{diagram:pathliftpathrev}
\xymatrix{
1 \ar[d] \ar[r] & A^\I \ar[d]^{A^i} \\
\I \ar[r] \ar@{..>}[ru] & A^{\partial\I}
}
\end{equation*}
But this is just the exponential transpose of the diagram \eqref{diagram:boxfillingrev}, where the corresponsing transposed maps are as indicated:
\begin{equation}\label{diagram:pathliftpathrev}
\xymatrix{
1 \ar[d] \ar[r]^{a'} & A^\I \ar[d]^{A^i} \\
\I \ar[r]_{b'} \ar@{..>}[ru]^{c'} & A^{\partial\I}
}
\end{equation}
\end{proof}

The foregoing can be generalized to higher dimensions as follows (see \cite{awodey:cubes}):

\begin{proposition}\label{prop:main}
For any cubical set $X$ and any $n\geq 1$,  the canonical map $X^\I \to X\times X$ has $n$-box filling iff $X$ has $(n+1)$-box filling.
\end{proposition}

Since $1$-box filling in any object $X$ is trivial, we conclude that a cubical set $X$ has $n$-box filling for all $n\geq 1$ just in case $X^\I \to X\times X$ does.  In this way, we are led to make the following definition.

\begin{definition}
A map of cubical sets $f : Y\to X$ is a \emph{Kan fibration} if it has $n$-box filling for all $n\geq 1$.  
A cubical set $X$ is a \emph{Kan complex} if the map $X\to 1$ is a Kan fibration.  
\end{definition}

\begin{remark}
Since the composition of two maps with $n$-box filling clearly also has $n$-box filling, if an object $A$ has $(n+1)$-box filling then its pathobject $A^\I$ has only $n$-box filling.  Thus if we want all types to have \Id-types, and if we have $\Sigma$-types so that we can form $\id{A} = \sum_{x,y:A}\id{A}(x,y)$,  then we are led to take as the types those cubical sets that have $n$-box filling for \emph{all} $n$, i.e.\ the Kan complexes, and as the dependent types the Kan fibrations.  It seems remarkable that the constraints imposed by modelling \Id-types have led us directly to the definition of Kan complex and Kan fibration!
\end{remark}

\subsection{Uniformity}

We could now use Quillen's small object argument to make a WFS in which the \R-maps are the Kan fibrations, by taking as a generating set of left maps all the open box inclusions.  Since we want an \emph{algebraic} WFS, however, we shall instead use a refinement of the small object argument due to R.~Garner, and appropriate to an algebraic notion of fibration (see \cite{Garner}).

\begin{definition}  A \emph{uniform Kan complex} is a cubical set $X$ equipped with the following structure:
\begin{enumerate}
\item  For each open box $\sqcup^n \mono \I^n$, each $k\geq 1$, and each map $b: \I^k \times\sqcup^n\to X$, there is given an extension $\phi(b) : \I^k\times\I^n \to X$ of $b$ along the product map $\I^k\times\sqcup^n \mono \I^k\times\I^n$.
\begin{equation}\label{diagram:uniformsimple}
\xymatrix{
\I^k\times \sqcup^n\ar@{>->}[d] \ar[r]^-{b}  & X\\
 \I^k \times \I^n \ar@{..>}[ru]_{\phi(b)} &
}
\end{equation}
\item The chosen extensions $\phi(b)$ are natural in $\I^k$, in the sense that for each map of cubes $\alpha : \I^j \to \I^k$, one has
\[
\phi(b\circ(\alpha\times 1))\ =\ \phi(b)\circ (\alpha\times 1),
\]
as indicated in the following commutative diagram.
\begin{equation}\label{diagram:uniformnatural}
\xymatrix{
\I^j\times \sqcup^n \ar@{>->}[d]  \ar[rr]^{\alpha\times 1} 
	&& \I^k\times \sqcup^n\ar@{>->}[d] \ar[r]^-{b} 
		& X\\
 \I^k \times \I^n \ar@{..>}[urrr]|{\phi(b(\alpha\times 1))\ \ } \ar[rr]_{\alpha\times 1} 
 	&&
	 \I^j \times \I^n  \ar@{..>}[ru]_-{\phi(b)} 
		&
}
\end{equation}
\item The foregoing also holds for all lower open boxes $\sqcap^n \mono \I^n$.
\end{enumerate}
\end{definition}

The uniform Kan condition turns the box filling \emph{property} of an object $X$ into an explicitly given \emph{structure} $(X, \phi)$ on $X$, namely a natural choice of fillers $\phi(b)$ for all (generalized) open boxes $b : \I^k\times \sqcup^n \to X$.  The generalization to maps is straightforward:

\begin{definition}[cf.~\cite{BCH}]
A \emph{uniform Kan fibration} is a map $f : Y\to X$ equipped with the following structure: 
\begin{enumerate}

\item  For each open box $\sqcup^n \mono \I^n$, each $k\geq 1$, and each (outer) square of the form
\begin{equation}
\xymatrix{
\I^k \times \sqcup^n \ar@{>->}[d] \ar[r]^-{b} & Y \ar[d]^{f} \\
\I^k \times \I^n \ar[r]_-{a} \ar@{.>}[ru]|-{\phi(a,b)} & X
}
\end{equation}
with the product map $\I^k\times\sqcup^n \mono \I^k\times\I^n$ on the left, there is given a diagonal filler $\phi(a,b)$ as indicated.

\item These chosen fillers are natural in $\I^k$, in the sense that for any $\alpha : \I^j \to \I^k$, one has
\[
\phi(a,b)\circ(\alpha\times 1) = \phi(a\circ (\alpha\times 1), b\circ (\alpha\times 1))\,,
\]
as indicated below, in which $\phi' = \phi(a\circ(\alpha\times 1), b\circ(\alpha\times 1))$.
\begin{equation}\label{diagram:coherencefillers}
\xymatrix{
\I^j \times \sqcup^n \ar@{>->}[d] \ar[rr]^{\alpha\times 1} && \I^k \times \sqcup^n  \ar@{>->}[d] \ar[rr]^-{b} && \ar[d]^{f} T \\
\I^j \times \I^n \ar[rr]_{\alpha\times 1} \ar@{..>}[rrrru]|{\phi'\ } && \I^k \times \I^n \ar[rr]_-{a} \ar@{.>}[rru]|-{\phi(a,b)}  && X
}
\end{equation}

\item The foregoing also holds for all lower open boxes $\sqcap^n \mono \I^n$.

\end{enumerate}
\end{definition}

Clearly, $X$ is a uniform Kan complex just in case $X\to 1$ is a uniform Kan fibration

\begin{theorem}\label{thm:AWFS1}
There is an algebraic weak factorization system $(\LL,\RR)$ on $\cSet$ for which the $\RR$-algebras are exactly the uniform Kan fibrations.
\end{theorem}

This is a direct application of Garner's small object argument \cite{Garner} (see also \cite{GS} for a related development). We will nonetheless sketch the details of the argument in the next section, because some modifications of it will be required thereafter.  

\section{The AWFS of uniform box filling}

In order to make an AWFS $(\LL,\RR)$ on $\cSet$ in which the $\RR$-algebras are the uniform Kan fibrations, we will first turn the description of uniform box filling structure on a map $f : X\to Y$ into that of an algebra structure for a suitable pointed endofunctor $T : \cSet/Y\to\cSet/Y$.  We then use a theorem of Kelly \cite{Kelly} to produce the algebraically free monad $T_\infty$ on $T$, which will be the $\RR$ part of the desired AWFS.  To simplify the exposition, we consider only the case where $Y=1$, i.e.\ Kan complexes rather than Kan fibrations, but the general case is entirely analogous.  Thus we are constructing the ``free uniform Kan complex'' on an arbitrary cubical set $X$, which may be regarded as a notion of $\infty$-groupoid.  The analogy to the construction of the free groupoid on a graph is a useful one to bear in mind.

Consider first a box filling structure $\phi$ on a cubical set $X$.  For each open $n$-box $i^n:\sqcup^n \mono \I^n$, each $k$, and for each (generalized) open box $b: \I^k \times\sqcup^n\to X$ in $X$, there is given an extension $\phi(b) : \I^k\times\I^n \to X$ to a $(k+n)$-cube in $X$, naturally in $\I^k$:
\[
\xymatrix{
\I^k\times \sqcup^n\ar@{>->}[d]_{1\times i^n\ } \ar[r]^-{b}  & X\\
 \I^k \times \I^n \ar@{..>}[ru]_{\phi(b)} &
}
\]
Transposing, we obtain an assignement $b \mapsto \phi(b)$ of the form:
\[
\xymatrix{
\I^k \ar[r]^{b}  \ar@{..>}[rd]_{\phi(b)}  & X^{ \sqcup^n}\\
 & \ar[u]_{X^{i^n}}  X^{ \I^n}\,,
}
\]
such that $X^{i^n}\circ\phi(b) = b$, again naturally in $\I^k$.

But since the cubes $\I^k$ are the representables, by Yoneda the assignment $\phi$ is given by composing with a unique natural transformation $\phi : X^{\sqcup^n}\to X^{\I^n}$, which is a section of $X^{i^n}$,
\[
\xymatrix{
  \ar@/{}_{1pc}/@{..>}[d]_{\phi} X^{ \sqcup^n}\\
  \ar[u]_{X^{i^n}}  X^{ \I^n}\,,
}
\]

Transposing by $\I^n$, we obtain the equivalent diagram:
\[
\xymatrix{
 X^{\sqcup^n}\times\sqcup^n \ar@{>->}[dd]_-{1\times i^n} \ar[rr]^-{\eval} && X \\
&& \\
 X^{\sqcup^n}\times\I^n \ar@{..>}[uurr]_{\phi} &&
}
\]

Finally, forming the pushout of the span formed by $1\times i^n$ and the evaluation map $\eval^n$,
\begin{equation}\label{diagram:defTn}
\xymatrix{
 X^{\sqcup^n}\times\sqcup^n \ar@{>->}[dd]_{1\times i^n} \ar[rr]^-{\eval^n} && X \ar[dd]^{t^n_X} \ar[r]^{=} & X\\
&&& \\
 X^{\sqcup^n}\times\I^n \ar[rr] && {\pocorner} T^n(X) \ar@{..>}[uur]_{\phi} &
}
\end{equation}
we obtain an object $T^n(X)$ with a map $t^n_X : X\to\ T^n(X)$ such that retractions $\phi$ of $t^n_X$ correspond bijectively to $n$-box filling structures on $X$.  

Since $T^n$ is clearly functorial in $X$, we have an endofunctor $$T^n : \cSet\to\cSet$$ with an (evidently natural) point $t^n : 1 \to T^n$.  Summing over all dimensions $n$, and again forming the pushout, 
\begin{equation}\label{diagram:defT}
\xymatrix{
{\coprod_{n}}X^{\sqcup^n}\times\sqcup^n \ar@{>->}[dd]_{{\coprod_{n}}1\times i^n\ } \ar[rr]^-{[\eval^n]} && X \ar[dd]^{t_X} \ar[r]^{=} & X\\
&&& \\
{\coprod_{n}}X^{\sqcup^n}\times\I^n \ar[rr] && {\pocorner} T(X) \ar@{..>}[uur]_{\phi} &
}
\end{equation}
we obtain an endofunctor $T : \cSet \to \cSet$ with a point $t : 1\to T$, 
the algebras for which are box filling structures on $X$, but now for all dimensions $n$.  It is easy to see that $(T,t)$-algebra homomorphisms  $h : (X,\phi) \to (Y, \psi)$ correspond exactly to maps $h : X\to Y$ that preserve the box filling structures, in the obvious sense.  Summarizing, we have shown:
\begin{proposition}
The category of $(T,t)$-algebras for the pointed endofunctor $$T : \cSet \to \cSet$$ defined by \eqref{diagram:defT} is isomorphic to the category of uniform Kan complexes.
\end{proposition}

Next we use a method due to Kelly (see \cite{Garner,Kelly}) to construct the \emph{algebraically free monad} $(T_\infty, t_\infty, \mu)$ from the pointed endofunctor $(T,t)$.  This monad comes with a natural map $\eta : T\to T_\infty$ such that $t_\infty = \eta\circ t$, inducing a comparison functor, $$\eta^* : (T_\infty, t_\infty, \mu)\mbox{-}\mathrm{Alg} \to (T,t)\mbox{-}\mathrm{Alg}\,,$$ which is an isomorphism of categories.  Here $(T_\infty, t_\infty, \mu)\mbox{-}\mathrm{Alg}$ is the category of algebras for the monad, while $(T,t)\mbox{-}\mathrm{Alg}$ is the category of algebras for the pointed endofunctor.  We henceforth refer to these more briefly as $T_\infty\mbox{-}\mathrm{Alg}$ and $T\mbox{-}\mathrm{Alg}$ respectively.

\begin{lemma}\label{lem:compact}
The endofunctor $T : \cSet \to \cSet$ preserves $\omega$-colimits.
\end{lemma}
\begin{proof}
The pathobject functor $X^\I$ preserves all colimits, since it has a right adjoint by Corollary \ref{cor:rightadjpath}; hence so do all the functors $X^{\I^n}$.  Since the open boxes $\sqcup^{n}$ are finite colimits of $(n-1)$-cubes, the functors $X^{\sqcup^{n}}$ are finite limits of functors of the form $X^{\I^{n-1}}$, and therefore preserve filtered colimits.  Therefore $T$ is a pushout of functors that preserve $\omega$-colimits.
\end{proof}

\begin{proposition}
There is a monad $\Tinf$ on $\cSet$ with a map of pointed endofunctors $T\to T_\infty$ inducing an isomorphism,
\[
T_\infty\mbox{-}\mathrm{Alg}\ \cong\ T\mbox{-}\mathrm{Alg}\,.
\]  
\end{proposition}

\begin{proof}(sketch, cf.\ \cite{Garner})
To construct $\Tinf(X)$ for any cubical set $X$,  begin with $t_X :X\to TX$ and  consider 
\[
\xymatrix{
TX  \ar@<.8ex>[r]^{Tt_X}  \ar@<-.8ex>[r]_{t_{TX}} & T^2X\,,
}
\]
which need not agree.  Take the coequalizer $c$ to obtain a new object, called $T_2X$,
\[
\xymatrix{
TX  \ar@<.8ex>[r]^{Tt_X}  \ar@<-.8ex>[r]_{t_{TX}} & T^2X \ar[r]^c & T_2X\,.
}
\]
Now let $T_1X = TX$ and $t_1 = t : X\to T_1X$ and $c_1 = c : T^2X\to T_2X$, and continue as follows:
\[
\xymatrix@=3em{
TX  \ar@<.8ex>[r]^{Tt_X}  \ar@<-.8ex>[r]_{t_{TX}} & T^2X \ar[rd]^{c_1} \ar[r]^{Tt_1} & TT_2X \ar[rd]^{c_2} \ar[r]^{Tt_2} & TT_3X \ar[r] & \dots\\
X \ar[u]^{t} \ar[r]_{t_1} & T_1X \ar[u]_{t_{T_1}} \ar[r]_{t_2} & T_2X \ar[u]_{t_{T_2}} \ar[r]_{t_3} & T_3X \ar[u]_{t_{T_3}} \ar[r]_{t_3} & \dots
}
\]
At each successive step, $t_{n+1} := c_n\circ t_{T_n}$ and $c_{n+1} := \mathrm{coeq}(Tt_n, t_{T_{n+1}}\circ c_n)$.

Finally, set $\Tinf(X) := \varinjlim_{n}T_n(X)$ and let $(t_\infty)_X: X\to \Tinf(X)$ be the canonical map.  It follows from the foregoing lemma that $t : \Tinf(X) \cong T(\Tinf(X))$, and so $\Tinf$ is the (algebraically) free monad on $T$.
\end{proof}

\begin{remark}\label{rem:AWFS1}
This essentially completes the proof of Theorem \ref{thm:AWFS1}: there is an AWFS $(\LL,\RR)$ on $\cSet$ in which the $\RR$-algebras are the uniform Kan fibrations.  For an object $X$, a uniform Kan structure $\phi$ on $X$ is the same thing as a $T$-algebra structure map $\phi : TX\to X$, and by the foregoing, these correspond precisely to $\Tinf$-algebra structures $\phi:\Tinf(X) \to X$.  Moreover, the free $T$-algebra on $X$ is simply $\Tinf(X)$, with the isomorphism $t^{-1} : T\Tinf(X) \cong \Tinf(X)$ as its $T$-structure map.

More generaly, for any map $f:Y\to X$ of cubical sets, we have that  $\RR(f) = \Tinf(f) : \tilde{Y} \to X$ is the ``free uniform Kan fibration'' on $f$, as indicated in
\begin{equation}
\xymatrix{
Y \ar[rrd]_{f} \ar[rr]^{(t_\infty)_f} && \ar[d]^{\Tinf(f)} \tilde{Y} \\
&& X
}
\end{equation}
while $\LL(f) = (t_\infty)_f : Y \to \tilde{Y}$ is the unit at $f$ of the $\Tinf$ monad.  The object $E(f) = \tilde{Y}$ is simply the domain of $\RR(f)$, constructed as a colimit.
\end{remark}

\subsection{Connections}

For the model of \Id-types, we need the following fact, which is easily checked along the lines of Proposition \ref{prop:pathliftboxfill}.

\begin{lemma}
If $A$ is a uniform Kan complex, then the canonical map from the pathobject $A^\I \to A\times A$ is a uniform Kan fibration.
\end{lemma}

Now let $A$ be (uniform) Kan and consider the canonical pathobject factorization $A\stackrel{r}{\to} A^\I \stackrel{p}{\to} A\times A$ of the diagonal.  By the foregoing lemma, the second factor $p : A^\I \to A\times A$ is also Kan, and so we just need the ``constant path'' map $r : A \to A^\I$ to be an $\LL$-map.  Specifically, it should have a coalgebra structure for the copointed endofunctor
\[
\LL : A/\cSet \longrightarrow A/\cSet\,,
\]
which takes maps $f : A\to X$ to their first factors $\LL(f): A \to E(f)$.  According to the construction of the AWFS just given (see Remark \ref{rem:AWFS1}), an $\LL$-coalgebra structure map for $r:A\to A^\I$ is therefore a diagonal filler for the following square.
\begin{equation}\label{diagram:reflL}
\xymatrix{
A \ar[d]_{r} \ar[r]^{\LL(r)} & \ar[d]^{\RR(r)} E(r) \\
A^\I \ar@{..>}[ru] \ar[r]_{=} & A^\I
}
\end{equation}
where $\RR(r):E(r) \to A^\I$ is the free fibration on $r : A \to A^\I$.  Since $\RR(r)$ is an $\RR$-algebra, this is exactly an instance of \Id-elimination, as in \eqref{diag:idelim}.

\begin{proposition}\label{prop:liftingnotnormal}
Let $A$ be a Kan complex, $\pi : B \to A^\I$ a Kan fibration, and $b : A\to B$ a map making the following outer square commute.
\begin{equation}\label{diagram:IdElim}
\xymatrix{
A \ar[d]_{r} \ar[r]^{b} & \ar[d]^{\pi} B \\
A^\I \ar@{..>}[ru]_{j} \ar[r]_{=} & A^\I
}
\end{equation}
Then there is a diagonal map $j$ as indicated making the lower triangle commute, $\pi\circ j = 1$.
\end{proposition}

For the proof we shall require the following:

\begin{definition}
A \emph{connection} on a cubical set $X$ is a map
\[
c : X^\I \to X^{\I\times\I}
\]
such that for each $n$-cube $(a: a_0\to a_1)$ in $X^\I$, the $(n+1)$-cube $c(a)$ has the form:
\begin{equation}\label{diagram:connection}
\xymatrix{
a_0 \ar[d]_{ra_0} \ar[r]^{ra_0} & \ar[d]^{a} a_0 \\
a_0 \ar@{}[ru]|{c(a)} \ar[r]_{a'} & a_1
}
\end{equation}
where $r : X\to X^\I$ is the degeneracy (i.e.\ the ``constant path'').

The connection is called \emph{strict} if in the above we always have $a'=a$.  It is called \emph{normal} if $c(r(x)) = r(r(x))$ for all $n$-cubes $x$ in $X$.
\end{definition}

\begin{lemma}
Every Kan complex  $A$ has a connection $c : A^\I \to A^{\I\times\I}$.
\end{lemma}
\begin{proof}
Use the box filling in $A$.
\end{proof}

\begin{proof} (Proposition)
We shall use the uniform Kan structure on $A$ and $\pi:B\to A^\I$ to construct a section $j:A^\I \to B$ of $\pi$.  Since we know that  $\pi$ has path-lifting, it will suffice to have, for every $n$-box $a : a_0 \to a_1$ in $\pA$, a path of the form $c(a) : r(a_0) \to a$ (naturally in $n$).  For then we can take $b(a_0)$ as a lift of $r(a_0)$ to get $j(a)$ as the transport
\[
j(a) := c(a)_*(b(a_0))
\]
of $b(a_0)$ along the path $c(a)$.  We then have $\pi(j(a)) = \pi(c(a)_*(b(a_0))) = a$ by the definition of transport.  

The required path $c(a) : r(a_0) \to a$, for any $a$, is provided by the connection on $A$, as given by the lemma.
\end{proof}

Now consider the top triangle in \eqref{diagram:IdElim}, i.e.\ $jr=b$, as is required in order to have a diagonal filler in \eqref{diagram:reflL}.  By definition, we have $j(a) := c(a)_*(b(a_0))$, and so $j(rx) = c(rx)_*(b(x))$.  Thus we will have $jr=b$ if the following two conditions are satisfied:
\begin{enumerate}
\item the connection $c$ is strict: $c(rx) = r(rx)$,
\item the transport in $\pi : B \to \pA$ along any degenerate path $r(x) : x\to x$ is trivial: $r(x)_* = 1$.
\end{enumerate}

Both of these conditions will follow from the following strengthening of the notion of a uniform Kan fibration.

\begin{definition}
A uniform Kan fibration $f:B\to A$, with structure $\phi$, is \emph{normal} if the filler assigned by $\phi$ to a degenerate open box is always degenerate.  Thus e.g.\  in the following outer commutative square,
\[
\xymatrix{
\sqcup^{n+1} \ar[dd]_{i\ }\ar[rr]^-{b} \ar[rd]^{\pi i} && B \ar[dd]^{f} \\
& \I^n \ar[ru]^{c} & \\
\I^{n+1} \ar[ru]^\pi \ar[rr]_a \ar@/{}_{1.5pc}/@{..>}[rruu]_{\phi(a,b)} && A
}
\]
if the open $(n+1)$-box $b$ is degenerate, $b=c\pi i$ for some $n$-cube $c$ in $B$ and projection $\pi: \I^{n+1} \to \I^{n}$ \emph{in the direction in which $\sqcup^{n+1}$ is open},  then the filler $\phi(a,b)$ is also degenerate, and specifically $\phi(a,b) = c\pi$.  

Generally, the assignment $\phi(a,b)$ must satisfy the following condition: for any generalized open box $i : \I^k\times\sqcup^{n+1}\to \I^k\times\I^{n+1}$ and any $(k+n)$-cube $c : \I^k\times \I^n \to B$ we have:
\[
\phi(fc\pi, c\pi i) = c\pi\,,
\]
where $\pi: \I^{n+1} \to \I^{n}$ is the projection in the direction in which $\sqcup^{n+1}$ is open, all as indicated below.
\[
\xymatrix{
\I^k\times\sqcup^{n+1} \ar[dd]_{i\ } \ar[rr] && B \ar[dd]^{f} \\
& \I^k\times\I^n \ar[ru]^{c} & \\
\I^k\times\I^{n+1} \ar[ru]^\pi \ar[rr] \ar@/{}_{2pc}/@{..>}[rruu]_{\phi} && A\,.
}
\]
\end{definition}

\begin{lemma}\label{lem:normalpath}
If $A$ is a normal Kan complex, then $\pA \to A\times A$ is a normal Kan fibration.
\end{lemma}
\begin{proof}
Straightforward.
\end{proof}

\begin{lemma}
Every normal Kan complex  $A$ has a normal connection $c : A^\I \to A^{\I\times\I}$.
\end{lemma}
\begin{proof}
Using normal box filling in $A$ to define the connection results in a normal connection.
\end{proof}

\begin{lemma}
If $f: Y\to X$ is a normal Kan fibration, then the transport operation along any degenerate path $r(x) : x\to x$ in $X$ is trivial.
\end{lemma}
\begin{proof}
Transport is determined by $1$-box filling, and so it is trivial if $f$ is normal.
\end{proof}

Using normal fibrations now allows a strengthening of Proposition \ref{prop:liftingnotnormal}.

\begin{proposition}\label{prop:liftingnormal}
Let $A$ be a normal Kan complex, $\pi : B \to A^\I$ a normal Kan fibration, and $b : A\to B$ a map making the following outer square commute.
\begin{equation}\label{diagram:IdElim}
\xymatrix{
A \ar[d]_{r} \ar[r]^{b} & \ar[d]^{\pi} B \\
A^\I \ar@{..>}[ru]_{j} \ar[r]_{=} & A^\I
}
\end{equation}
Then there is a diagonal filler $j$ as indicated, making both triangles commute, $\pi\circ j = 1$ and $jr=b$.
\end{proposition}
\begin{proof}
Again we set $j(a) := c(a)_*(b(a_0))$ and it just remains to check that $jr=b$.  But now we have:
\[
jr(x) = c(r(x))_*(b(x)) = r(r(x))_*(b(x)) = b(x)
\]
where the second equation is because the connection in $A$ is normal, and the third because the transport in $B\to\pA$ along $r(r(x))$ is trivial.
\end{proof}

\begin{remark}
On reflection, it should come as no surprise that the fibrations in our model must be normal in this sense in order to also model the \Id-computation rule: the syntactic model has the closely related property that the transport operation along any reflexivity term is (definitionally) trivial, in virtue of the \Id-computation rule.

Of course, to model \Id-types in general we require not only the factorizations of the form $A \to A^\I \to A\times A$ for Kan complexes $A$, but also all those of the form $A \to A^\I \to A\times_X A$, for Kan fibrations $A\to X$, at least for $X$ Kan, and where the exponential $A^\I$ is made in the slice category $\cSet/X$.  As elsewhere, we consider only the case $X=1$ for ease of exposition, but the general case holds as well.
\end{remark}

\section{Normalization}

Let $X$ be a Kan complex with uniform box-filling structure $\phi$, so that for any open box $b: \sqcup^n \to X$ there is an associated filler $\phi(b) : \I^n \to X$ extending $b$,
\[
\xymatrix{
\sqcup^n \ar[d]_{i^n} \ar[rr]^-{b} && X\\
\I^n \ar@{..>}[urr]_{\phi(b)} &&\,.
}
\]
Since we have symmetries of cubes, we may restrict attention to boxes $i : \sqcup^n \to \I^n$ that are open in the first dimension (in either direction).  The filling structure is then normal if, for any $(n-1)$-cube $c : \I^{n-1}\to X$, we have
\[
\phi(c\pi i) = c\pi\,,
\]
as in
\[
\xymatrix{
\sqcup^n \ar[dd]_{i^n} \ar[rr]^-{c\pi i} && X\\
	& \I^{n-1} \ar[ur]^c & \\
\I^n  \ar[ru]^{\pi} \ar@/{}_{2pc}/@{..>}[rruu]_{\phi(c\pi i)} &&\,,
}
\]
where $\pi : \I^n \to \I^{n-1}$ is the first projection, so that   $\pi i : \sqcup^n \to \I^{n-1}$ projects the open box onto its ``bottom''.  
Of course, the same condition applies also to generalized open boxes $b: \I^k\times\sqcup^n \to X$.

This condition can be reformulated equivalently by saying that the section $\phi$ in the diagram
\[
\xymatrix{
& X^{ \I^n } \ar[d]^{X^{i}}\\
X^{\sqcup^n} \ar[ru]^{\phi} \ar[r]_= & X^{\sqcup^n} 
}
\]
must also make the upper triangle in the following diagram commute:
\[
\xymatrix{
X^{\I^{n-1}} \ar[d]_{X^{\pi i}} \ar[r]^{X^\pi} & X^{ \I^n } \ar[d]^{X^{i}}\\
X^{\sqcup^n} \ar[ru]^{\phi} \ar[r]_= & X^{\sqcup^n} \,.
}
\]
Transposing, we obtain the condition:
\[
\xymatrix{
{X^{\I^{n-1}}\!\times\I^{n}\ } 	\ar[d]_-{X^{{\pi}i} \times 1\ } \ar[rr]^-{1\times\pi} && {X^{\I^{n-1}}\!\times\I^{n-1}\ } \ar[d]^{\eval} \\ 
X^{\sqcup^n}\times\I^{n}		\ar[rr]^-{\phi} 	&& X \\
X^{\sqcup^n}\times\sqcup^n 	\ar[u]^-{1\times i\ } 	\ar[rr]_-{\eval} && X \ar[u]_{=} \,.
}
\]
Or, equivalently,
\[
\xymatrix{
{X^{\I^{n-1}}\!\times\I^{n}\ } 	\ar[d]_-{X^{{\pi}i} \times 1\ } \ar[rr]^-{\eval(1\times\pi)} && X \ar[d]^{=} \\ 
X^{\sqcup^n}\times\I^{n}		\ar[rr]^-{\phi} 	&& X \\
X^{\sqcup^n}\times\sqcup^n  	\ar[u]^-{1\times i\ } \ar[rr]_-{\eval} && X \ar[u]_= \,.
}
\]
Using a coproduct to ``fold'' along $\phi$, we obtain:
\[
\xymatrix{
(X^{\sqcup^n}\times\sqcup^n) + (X^{\I^{n-1}}\!\times\I^{n}) \ar[dd]_-{[1\times i,\,X^{{\pi}i} \times 1]\ } \ar[rrr]^-{[\eval,\,\eval(1\times\pi)]} 
									&&& X \\ 
									&&&\\
X^{\sqcup^n}\times\I^{n}		\ar[rrruu]_-{\phi} 	&&& 
}
\]
Pushing out as in \eqref{diagram:defTn}, we have:
\begin{equation}\label{diag:defTdotn}
\xymatrix{
(X^{\sqcup^n}\times\sqcup^n) + (X^{\I^{n-1}}\!\times\I^{n}) \ar[dd]_-{[1\times i,\,X^{{\pi}i} \times 1]\ } \ar[rr]^-{[\eval,\,\eval(1\times\pi)]} 
									&& X \ar[dd]^{\dot{t}^n_X} \ar[r]^{=} & X\\
									&&&\\
X^{\sqcup^n}\times\I^{n}	\ar[rr] && {\pocorner} \dot{T}^n(X) \ar@{..>}[uur]_{\phi} & \,.
}
\end{equation}
This last description provides an object $\dot{T}^n(X)$ with a map $\dot{t}^n_X : X\to\ \dot{T}^n(X)$ such that retractions $\phi$ of $\dot{t}^n_X$ correspond uniquely to \emph{normal} $n$-box filling structures on $X$.  

Finally, as in \eqref{diagram:defT}, we can sum over all $n$ to obtain a pointed endofunctor $\dot{T} :\cSet\to\cSet$, the algebras for which correspond to \emph{normal} Kan complexes.

\begin{theorem}\label{thm:AWFS2}
There is an algebraic weak factorization system $(\LL,\RR)$ on $\cSet$ for which the $\RR$-algebras are exactly the normal Kan fibrations. \end{theorem}

\begin{proof}(sketch)
We have already constructed a pointed endofunctor $$\dot{T} :\cSet\to\cSet,$$ the algebras for which are the normal Kan complexes $X\to 1$.   The general construction for normal Kan fibrations $f:X\to Y$ is entirely analogous.  It just remains to show that $\dot{T}$ preserves $\omega$-colimits, so that Kelly's construction of the algebraically free monad $\dot{T}\to \dot{T}_\infty$ again applies. But this follows just as in Lemma \ref{lem:compact}, by inspecting the components in the defining diagram \eqref{diag:defTdotn}.
\end{proof}

\begin{proposition}
For any normal Kan complex $A$, the canonical pathobject factorization $A\to A^\I\to A\times A$ is an $(\LL,\RR)$ factorization.
\end{proposition}

\begin{proof}
The second factor $A^\I\to A\times A$ is normal by Lemma \ref{lem:normalpath}. The first factor $r:A\to A^\I$ is an $\LL$-map iff there is a diagonal filler for the square in \eqref{diagram:reflL}.  This is provided by Proposition \ref{diagram:IdElim}.
\end{proof}

With this, we have reached our goal: 
\begin{corollary}
The category $\cSet$ of cartesian cubical sets admits an interpretation of type theory with \Id-types based on an algebraic weak factorization system, in which the \Id-types are taken to be path-objects $\id{A} = A^\I$ satisfying the standard \Id-elimination and computation rules.  
\end{corollary}

\section{An application: Factorization}

The factorization of an arbitrary map $f:X\to Y$ given by the $\dot{T}_\infty$ monad is not easy to describe directly.  But in case $X$ and $Y$ are both Kan, there is another factorization that can more easily be described explicitly: the ``graph'' or ``homotopy image'' factorization. This is a well-known construction from homotopy theory, but the proof that it works in this case is not trivial.  We can use the interpretation of type theory to give an alternative proof (or at least a different perspective on the usual proof).

Let $A$ and $B$ be (normal, uniform) Kan complexes, and take any map $f:A\to B$.  We will construct a factorization 
\[
\xymatrix@=3em{
A \ar[r]^{i} \ar[rd]_{f}  & \ar[d]^{\widetilde{f}} \widetilde{A} \\
& B
}
\]
with $i\in\L$ and $\widetilde{f}\in\R$.

Consider the following pullback diagram:
\[
\xymatrix@=3em{
Pf \ar[r] \ar[d] \pbcorner  \ar@/{}_{2pc}/@{..>}[dd]_{\widetilde{f}} & B^\I \ar[d] \\
A\times B \ar[r]_{f\times 1} \ar[d]^{\pi_2} & B\times B \\
B 
}
\]
Take $\widetilde{A} = Pf$ and $\widetilde{f}$ the indicated composite, which, note, is an $\R$-map, since each factor is, by the assumption that both $A$ and $B$ are Kan.

For the factorization $i: A \to \widetilde{A}$, take the unique map determined as in:
\[
\xymatrix@=3em{
A \ar@{..>}[rd]^{i} \ar[r]^{f} \ar@/{}_{2pc}/[rddd]_{f} \ar@/{}_{1pc}/[rdd]^{(1,f)}  & B \ar[rd]^{r_B} & \\
		& {Pf} \ar[r] \ar[d] \pbcorner & B^\I \ar[d] \\
		& A\times B \ar[r]_{f\times 1} \ar[d]^{\pi_2} & B\times B \\
		& B 
		}
\]

As a dependent type, $Pf \to B$ is constructed as follows:
$$
\begin{prooftree}
\[ x: A \vdash f(x) : B, \qquad  y,y' : B \vdash \id{B}(y, y')
\justifies
x: A, y: B \vdash \id{B}(fx, y)
\]
\justifies
y: B \vdash \sum_{x:A}\id{B}(fx, y)
\end{prooftree}
$$
Under propositions-as-types, this is the ``image'' $$\{y:B\ |\ \exists x:A.\, f(x) = y\} \to B.$$
In the factorization 
\[
\xymatrix@=3em{
A \ar[r]^{i} \ar[rd]_{f}  & \ar[d]^{\widetilde{f}} Pf \\
& B
}
\]
the term $i$ has the form
\begin{align*}
x: A \vdash i(x) &: Pf(fx),\\
			&= \sum_{x:A}\id{B}(fx, fx)\,.
\end{align*}
Indeed, it is given by setting
\[
i(x) := (x, \refl(fx))\,.
\]

We know that $\widetilde{f} : Pf \to B$ is a (normal, uniform) Kan fibration, because it is a composite of pullbacks of such, as already noted; but we need to see that $i : A\to Pf$ as just defined is an $\L$-map. Consider 
\[
Pf = \sum_{x:A}\sum_{y:B}\id{B}(fx, y)
\] 
and take any $C\to Pf$ in $\R$, i.e.:
\[
x: A, y:B, z:\id{B}(fx,y) \vdash C(x,y,z)\,.
\]
Consider the lifting problem
\[
\xymatrix@=3em{
A \ar[r]^{d} \ar[d]_{i} & C \ar[d] \\
Pf \ar[r]_{=} \ar@{..>}[ru]_{\phi} & Pf\,.
}
\]
Thus we are given
\[
x: A \vdash d(x) : C\big(x,fx,\refl(fx)\big)\,,
\]
and we seek
\[
x: A, y:B, z:\id{B}(fx,y) \vdash \phi(x,y,z) : C(x,y,z)\,.
\]
such that:
\[
\phi(x,fx,\refl(fx)) = d(x)\,.
\]
But this is a known inference, for which see \cite{GG}, Lemma 11.

There is, of course, also a purely algebraic proof of this factorization; cf.~\cite{GvdB}, Proposition 6.1.4.

\subsection*{Acknowledgements}
For support during my stay in Stockholm, I am grateful to the University of Stockholm, the Arrhenius Foundation, and  the US Air Force Office of Sponsred Research through MURI Grant FA9550-15-1-0053.


\end{document}